\documentclass[reqno,12pt]{amsart}
\usepackage{amsfonts}
\usepackage{}
\usepackage{mathbbold}
\usepackage{amscd}
\usepackage{amssymb}
\usepackage{amsmath}
\usepackage{amsfonts}
\usepackage{enumerate}
\usepackage{graphicx}
\usepackage{epic}
\usepackage{overpic}
\usepackage{galois}
\usepackage{subfig}
\usepackage{caption}
\usepackage{varioref}

\newtheorem{thm}{Theorem}[section]

\newtheorem{lem}[thm]{Lemma}
\newtheorem{cor}[thm]{Corollary}
\theoremstyle{definition}
\newtheorem{Def}[thm]{Definition}

\newtheorem*{LBT}{Lower Bound Theorem}
\newtheorem{que}[thm]{Question}

\newcommand{\kk}{\mathbf{k}}

\newcommand{\ZZ}{\mathcal {Z}}

\newcommand{\Ss}{\mathcal {S}}

\newcommand{\VV}{\mathcal {V}}

\newcommand{\MM}{\mathcal {M}}

\def\w{\widetilde}

\def\b{\w}
\begin{document}
\title[$B$-Rigidity of flag $2$-spheres without $4$-belt]{$B$-Rigidity of flag $2$-spheres without $4$-belt}
\author[F.~Fan, J.~Ma \& X.~Wang]{Feifei Fan, Jun Ma and Xiangjun Wang}
\thanks{The first and third authors are supported by the National Natural Science Foundation of China (NSFC no. 11471167)}
\address{Feifei Fan, Academy of Mathematics and Systems Science, Chinese Academy of sciences, Beijing 100190, P.~R.~China}
\email{fanfeifei@mail.nankai.edu.cn}
\address{Jun Ma, School of Mathematical Sciences, Fudan University, Shanghai 200433, P.~R.~China}
\email{tonglunlun@gmail.com}
\address{Xiangjun Wang, School of Mathematical Sciences, Nankai University, Tianjin 300071, P.~R.~China}
\email{xjwang@nankai.edu.cn}
\maketitle
\begin{abstract}
Associated to every finite simplicial complex $K$, there is a
moment-angle complex $\mathcal {Z}_{K}$. In this paper, we use some algebraic invariants to solve the $B$-rigidity problem for some special simplicial compelexes.
\end{abstract}
\section{Introduction}
In this paper, we assume $K$ is a simplcial complex with vertex set $[m]=\{v_1,\dots,v_m\}$, and $\kk$ is a field unless otherwise stated. Denote by $MF(K)$ the set of all missing faces of $K$.

Associated to every finite simplicial complex $K$, there is a
moment-angle complex $\mathcal {Z}_{K}$. The study of the moment-angle complexes is one of the main problems in toric topology, and one of the main tool to study this object is the cohomology of $\mathcal {Z}_{K}$. Around its cohomology, there is a classical problem proposed by Buchstaber in his lecture note \cite{B08}.
\begin{que}
Let $K_1$ and $K_2$ be two simplicial complexes, and let $\mathcal {Z}_{K_1}$ and
$\mathcal {Z}_{K_2}$ be their respective moment-angle complexes. When a graded ring
isomorphism $H^*(\mathcal {Z}_{K_1};\kk)\cong H^*(\mathcal {Z}_{K_2};\kk)$
implies a combinatorial equivalence $K_1\approx  K_2$?
\end{que}
Let us call the simplicial complex giving the positive answer to the
question \emph{$B$-rigid over $\kk$} (if $\kk=\mathbb{Z}$, simply refer to it as \emph{$B$-rigid}).

Because the topology of a moment-angel complex $\mathcal {Z}_{K}$ is uniquely determined by its underling simplicial complex $K$, the answer to this question can be a guideline for the following important problem in toric topology.

\begin{que}\label{que:1}
Suppose $\mathcal {Z}_{K_1}$ and $\mathcal {Z}_{K_2}$ are two moment-angle manifolds such that
\[H^*(\mathcal {Z}_{K_1})\cong H^*(\mathcal {Z}_{K_2})\] as graded rings. Are $\mathcal {Z}_{K_1}$ and $\mathcal {Z}_{K_1}$ homeomorphic?
\end{que}
In this paper we study the $B$-rigidity of a special kind of simplicial complexes: \emph{flag $2$-spheres without $4$-belt} by following the technique in \cite[\S7]{FW}.

\section{Definitions}
\begin{Def}
Given a subset $I\subseteq \mathcal [m]$, define $K_I\subseteq K$ to be the \emph{full sub-complex} of $K$ consisting of all simplices of $K$
which have all of their vertices in $I$, that is
\[K_I:=\{\sigma\cap I\mid \sigma\in K\}.\]
\end{Def}
\begin{Def}
A full subcomplex $K_I$ of $K$ is called an $n$-belt of $K$, if $K_I$ is a triangulation of $S^1$ with $n$ vertices.
\end{Def}
\begin{Def}
The \emph{face ring} (also known as the \emph{Stanley-Reisner ring}) of a simplicial complex $K$ is the quotient ring
\[\kk({K}):=\kk[v_1,\dots\,v_m]/\mathcal {I}_{K}\]
where $\mathcal {I}_K$ is the ideal generated by the monomials $v_{i_1}\cdots v_{i_s}$ for
which $\{v_{i_1},\dots,v_{i_s}\}$ does not span a simplex of $K$.
\end{Def}
To calculate the cohomology of $\mathcal {Z}_K$, Buchstaber and Panov \cite{BP02} proved the following
\begin{thm}[Buchstaber-Panov, {\cite[Theorem 4.7]{P08}}]
\begin{align*}
H^*(\mathcal {Z}_{K}; \kk)\cong \mathrm{Tor}^{*,\,*}_{\kk[m]}(\kk(K),\kk)\cong\bigoplus_{I\subseteq [m]} \w {H}^*(K_I;\kk)
\end{align*}
\end{thm}
The second isomorphism (as $\kk$-module) in the above theorem was firstly proved by Hochster \cite{H75}. Baskakov \cite{B02}
defined a natural multiplication structure on the $\kk$-module $\bigoplus_{I\subseteq [m]} \w {H}^*(K_I;\kk)$ so that it can be a ring isomorphism in the above theorem. 
It is induced by a canonical simplicial inclusion
\[\eta: K_{I\cup J}\hookrightarrow K_I*K_J,\quad I\cap J=\varnothing\]
and isomorphisms of reduced simplicial cochains
\begin{align*}
\mu :\w C^{p-1}(K_I;\kk)\otimes \w C^{q-1}(K_J;\kk)&\to \w C^{p+q-1}(K_I*K_J;\kk), \quad p,q\geq 0\\
\sigma\otimes\tau&\mapsto \sigma\cup \tau
\end{align*}
We call this ring the \emph{Baskakov-Hochster ring}.

Actually the formula given by Baskakov holds only up to a sign. Buchstaber and Panov indicated this defect and gave a correction in \cite{BP13}. 
F. Fan and X. Wang \cite[Theorem 2.12]{FW} gave a more explicit expression for this sign.

\begin{Def}
A simplicial complex $K$ is said to be \emph{Gorenstein* over $\kk$} if for any simplex $\sigma\in K$ (including $\sigma=\varnothing$)
\[\w H^i(\mathrm{link}_K\sigma;\kk)=
\begin{cases}
\kk&\quad\text{if}\  i=\mathrm{dim(link}_K\sigma);\\
0&\quad\text{otherwise.}
\end{cases}
\]
If $\kk=\mathbb{Z}$, then $K$ is simply called \emph{Gorenstein*}.
\end{Def}
It is easily verified that $K$ is Gorenstein iff $K$ is Gorenstein* over any field $\kk$. In particular, any simpicial sphere is Gorenstein*.
\begin{thm}[Avramov-Golod, {\cite[Theorem 3.4.5]{BH98}}]\label{thm:3}
An $(n-1)$-dimensional simplicial complex $K$ with $m$ vertices is Gorenstein* over a field $\kk$
if and only if $\mathrm{Tor}_{\kk[m]}(\kk(K),\kk)$
is a Poincar\'e algebra.
\end{thm}

\begin{Def}
Let $A$ be an algebra over a field $\kk$. Given a nonzero element $\alpha\in A$, if a $\kk$-subspace $V\subset A$ satisfies for any non-zero element $v\in V$, $v$ is a factor of $\alpha$ (i.e., there exists $u\in A$ such that $v\cdot u=\alpha$), then $V$ is called a \emph{factor space of $\alpha$ in $A$}. Denote by
$\mathcal {F}_\alpha$ the set of all factor spaces of $\alpha$. Define the \emph{factor index of $\alpha$} to be
\[\mathrm{ind}_A(\alpha):=\mathrm{max}\{\mathrm{dim}_\kk(V)\mid V\in \mathcal {F}_\alpha\}.\]

If $A=\bigoplus^d_{i=0}A^i$ is a graded $\kk$-algebra, and $\alpha\in A^j$ is a non-zero homogeneous element, if $V\subset A^k$ ($k\leq j$) is a factor space of $\alpha$ in $A$, then $V$ is called a \emph{$k$-factor space of $\alpha$ in $A$}. Denote by $\mathcal {F}_\alpha^k$ the set of all $k$-factor spaces of $\alpha$. The \emph{$k$-factor index of $\alpha$ in $A$}
is defined to be
\[\mathrm{ind}^k_A(\alpha):=\mathrm{max}\{\mathrm{dim}_\kk(V)\mid V\in \mathcal {F}_\alpha^k\}\]
\end{Def}

\begin{Def}
Let $R$ be a ring. For an element $r\in R$, the \emph{annihilator} of $r$ is defined to be
\[\text{ann}_R(r):=\{a\in R\mid a\cdot r=0\}.\]
\end{Def}
Apparently, if $A$ is an algebra over a field $\kk$, then for any element $\alpha\in A$, $\mathrm{ind}_A(\alpha)$ and $\mathrm{dim}_\kk(\mathrm{ann}_A(\alpha))$ are both algebraic invariants under isomorphisms.

\section{Results}
The main result of this paper is the following
\begin{thm}\label{thm:1}
Let $K$ be a flag $2$-sphere without $4$-belt. Then for any simplicial complex $K'$, if there is a graded isomorphism:
\[H^*(\ZZ_K)\cong H^*(\ZZ_{K'}),\]
then $K$ is combinatorially equivalent to $K'$.
\end{thm}
To prove this theorem, first we prove it holds for the case $K'$ is a simplical $2$-sphere. Second, we prove that for a simplicial complex $K'$, if $H^*(\ZZ_{K'};\kk)$ agree with some algebraic property of $H^*(\ZZ_K;\kk)$ with $K$ a flag $2$-sphere, then $K'$ itself must be a flag $2$-sphere.
We separate off the proof into several lemmas.
\begin{figure}[!ht]
\vspace{6pt}
\begin{overpic}[scale=0.35]{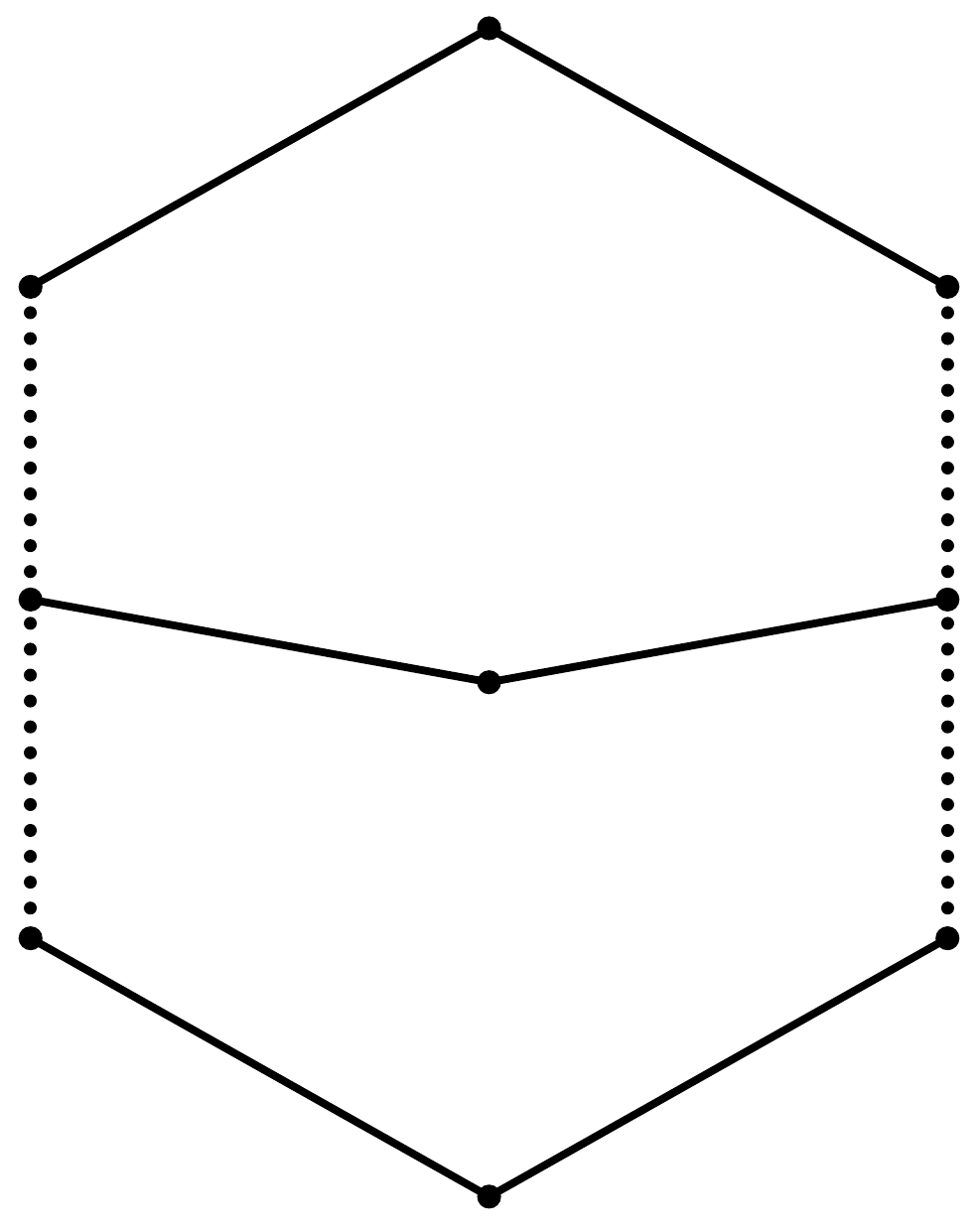}
\put(38,102){$v_1$}
\put(38,-6){$v_2$}
\put(38,38){$v_3$}
\put(-8,77){$u_1$}
\put(-8,22){$u_s$}
\put(-11,50){$u_{k_0}$}
\put(79,77){$w_1$}
\put(79,22){$w_t$}
\put(79,50){$w_{j_0}$}
\end{overpic}
\qquad\raisebox{12ex}{$\longrightarrow$}\qquad
\begin{overpic}[scale=0.35]{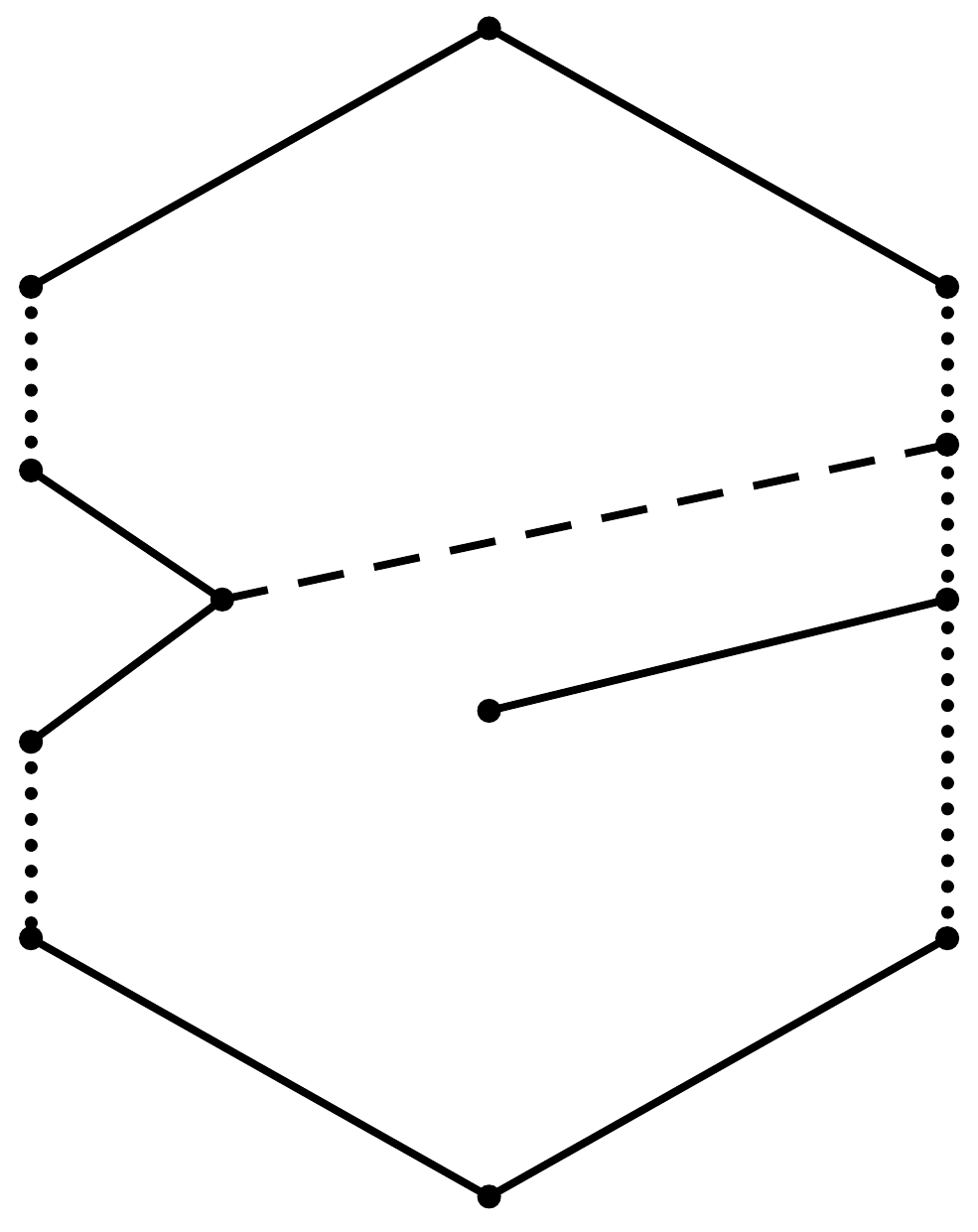}
\put(38,102){$v_1$}
\put(38,-6){$v_2$}
\put(38,36){$v_3$}
\put(-8,77){$u_1$}
\put(-8,22){$u_s$}
\put(79,77){$w_1$}
\put(79,22){$w_t$}
\put(79,48){$w_{j_0}$}
\put(79,62){$w_{j_1}$}
\end{overpic}
\vspace{6pt}
\caption{}\label{fig:1}
\end{figure}

\begin{lem}\label{lem:1}
Let $K$ be a flag $2$-sphere with vertex set $[m]$. Suppose $K$ has no $4$-belt. If $\{v_1,v_2\}\in MF(K)$ and $v_3$ is another vertex of $K$, then there exists $I\subset [m]\setminus\{v_3\}$ such that $\{v_1,v_2\}\subset I$, $K_I$ is a triangulation of $S^1$ and $\w H^0(K_{J})\neq0$, where $J=\{v_3\}\cup I\setminus\{v_1,v_2\}$.
\end{lem}
\begin{proof}
By \cite[Lemma 5.2]{FW}, there is a subset $I'\subset[m]\setminus\{v_3\}$, such that $\{v_1,v_2\}\subset I'$ and $K_{I'}$ is a triangulation of $S^1$. Clearly $K_{I'\setminus\{v_1,v_2\}}$ has two components $L_1$ and $L_2$ which are both the triangulation of $D^1$. Suppose the vertex sets of $L_1$ and $L_2$ are $U=\{u_1,\dots,u_s\}$ and $W=\{w_1,\dots,w_t\}$ resp. Set $U_0=\{v_1,v_2\}\cup U$, $W_0=\{v_1,v_2\}\cup W$.
If $\mathrm{link}_K(v_3)\cap L_1=\varnothing$ or $\mathrm{link}_K(v_3)\cap L_2=\varnothing$, then we get the desired $I=I'$. So we only consider the case
$\mathrm{link}_K(v_3)\cap L_1\neq\varnothing$ and $\mathrm{link}_K(v_3)\cap L_2\neq\varnothing$.

First we deal with the case that each of $\mathrm{link}_K(v_3)\cap L_1$ and $\mathrm{link}_K(v_3)\cap L_2$ has only one vertex, say $u_{k_0}$ and $w_{j_0}$ resp. (which are shown in Figure \ref{fig:1}). Apparently $K_{I'}$ separates $K$ into two disks
$K_1$ and $K_2$ with $K_{I'}$ as the common boundary. Suppose $v_3\in K_1$. Let $\mathcal {V}_{k_0}$ be the vertex set of $\mathrm{link}_{K_2}(u_{k_0})$. Then it is easy to see that
there exist $U'_0\subset U_0\setminus\{u_{k_0}\}$ and $\mathcal {V}_{k_0}'\subset \VV_{k_0}$ such that the full subcomplex restricted on $U_1=U'_0\cup \VV_{k_0}'$ ($\{v_1,v_2\}\subset U_1$) is still a triangulation of $D^1$.
Set \[\Gamma_1=\bigcup_{u\in \VV_{k_0}'}\mathrm{link}_{K}(u)\cap W.\]
If $\Gamma_1=\varnothing$, then put $I=U_1\cup W_0$. It is easily verified that $I$ satisfies the condition in the lemma. Otherwise, for $\Gamma_1\neq\varnothing$, since $K$ has no $4$-belt, we have $w_{j_0}\not\in\Gamma_1$ (otherwise $K_{(v_3,u_{k_0},v,w_{j_0})}$ would be a $4$-belt for some $v\in\mathrm{link}_{K_2}(u_{k_0})$).
For simplicity we may assume $\Gamma_1$ has only one vertex, say $w_{j_1}$ with $j_1<j_0$. Since $K$ has no $4$-belt, $v_3\not\in\mathrm{link}_{K_1}(w_{j_1})$ (otherwise $K_{(w_{j_1},v_3,u_{k_0},v)}$ would be a $4$-belt for some $v\in\mathrm{link}_{K_2}(u_{k_0})$). Let $\Ss_{j_1}$ be the vertex set of $\mathrm{link}_{K_1}(w_{j_1})$.
As before we can find $W_0'\subset W_0\setminus \{w_{j_1}\}$ and $\Ss_{j_1}'\subset \Ss_{j_1}$ such that the full subcomplex on $W_1=W_0'\cup \Ss_{j_1}'$ ($\{v_1,v_2\}\subset W_1$) is a triangulation of $D^1$. Set \[\Omega_1=\bigcup_{w\in \Ss_{j_1}'}\mathrm{link}_{K}(w)\cap (U_1\setminus\{v_1,v_2\}).\]
If $\Omega_1=\varnothing$, then put $I=W_1\cup U_1$, we get the desired $I$. Otherwise for simplicity we may assume $\Omega_1$ has only one vertex $u_{k_1}$ (clearly $k_1<k_0$). Let $\mathcal {V}_{k_1}$ be the vertex set of $\mathrm{link}_{K_2}(u_{k_1})$, then as before we can find $U'_1\subset U_1\setminus\{u_{k_1}\}$ and $\mathcal {V}_{k_1}'\subset \VV_{k_1}$ such that the full subcomplex on $U_2=U_1'\cup \VV_{k_1}'$ ($\{v_1,v_2\}\subset U_2$) is a triangulation of $D^1$.
Set \[\Gamma_2=\bigcup_{u\in \VV_{k_1}'}\mathrm{link}_{K}(u)\cap (W_1\setminus\{v_1,v_2\}).\]
If $\Gamma_2=\varnothing$, then put $I=U_2\cup W_1$, we still get the desired $I$. Otherwise, for simplicity we may assume $\Gamma_2$ has only one vertex $w_{j_2}$ (clearly $j_2<j_1$). Let $\Ss_{j_2}$ be the vertex set of $\mathrm{link}_{K_1}(w_{j_2})$.
As before we can find $W_1'\subset W_1\setminus \{w_{j_2}\}$ and $\Ss_{j_2}'\subset \Ss_{j_2}$ such that the full subcomplex on $W_2=W_1'\cup \Ss_{j_2}'$ ($\{v_1,v_2\}\subset W_2$) is a triangulation of $D^1$. Set \[\Omega_2=\bigcup_{w\in \Ss_{j_2}'}\mathrm{link}_{K}(w)\cap (U_2\setminus\{v_1,v_2\}).\]
If $\Omega_2=\varnothing$, then we get the desired $I=W_2\cup U_2$. Otherwise for simplicity we may assume $\Omega_2$ has only one vertex $u_{k_2}$ (clearly $k_2<k_1$).
Continuing this procedure if needed, then after finite step, we can actually get the desired $I$.

For the general case, we can remove the vertex of $\mathrm{link}_K(v_3)\cap L_1$ one by one by the procedure above.
\end{proof}

Since $H^*(\ZZ_K;\kk)$ is isomorphic to the Baskakov-Hochster ring $\bigoplus_{I\subseteq [m]} \w {H}^*(K_I;\kk)$, we do not distinguish these two rings whenever there is no confusion.
Define \[p_J:\bigoplus_{I\subseteq [m]} \w {H}^*(K_I;\kk)\to \w {H}^*(K_J;\kk)\] to be the projection homomorphism.

If $K$ is a flag complex, then for any missing face $\omega$ of $K$ (so $\omega$ contains only two vertices), $\w H^*(K_\omega)=\w H^*(S^0)=\kk$. Denote by $\b\omega$ a generator of this group.

\begin{lem}\label{lem:2}
Let $K$ be a flag $2$-sphere without $4$-belt. Suppose $MF(K)=\{\omega_1,\dots\omega_n\}$, and $\alpha=\sum_{i=1}^n r_i\cdot\b\omega_i,\, r_i\in \kk$. Define $\MM_\alpha=\{i\mid r_i\neq0\}$. If $|\MM_\alpha|\geq2$, then for each $i\in \MM_\alpha$ we have
\[\mathrm{dim}_\kk(\mathrm{ann}_R(\b \omega_i))>\mathrm{dim}_\kk(\mathrm{ann}_R(\alpha)),\]
where $R=H^*(\ZZ_K;\kk)$.
\end{lem}
\begin{proof}
In this proof, We omit the coefficient ring $\kk$ in the cohomology groups. Without loss of generality, we may assume $\alpha=\b\omega_1+\b\omega_2,$ and we only prove the inequality for $\b\omega_1$. View $R$ as a vector space over $\kk$, then we can find a complementary subspace $V_{\b\omega_1}$ of $\mathrm{ann}_R(\b w_1)$ in $R$ such that
\[\mathrm{ann}(\b \omega_1)\oplus V_{\b\omega_1}=R.\]
It is easy to see that for any $\beta\in V_{\b\omega_1}$, $\beta\b\omega_1\neq 0$. In fact, we can choose $V_{\b\omega_1}$ properly, such that $V_{\b\omega_1}$ has a basis $\{\beta_1,\dots\beta_s\}$ with the property that for each $i$, $\beta_i\in\w H^*(K_{I_i})$ for some $I_i\subset [m]$. This procedure can be realized as follows:
For each $I\subset[m]$, choose a complementary subspace $V_{\b\omega_1,I}$ of $\mathrm{ann}_R(\b w_1)\cap\w H^*(K_I)$ in $\w H^*(K_I)$, then $V_{\b\omega_1}$ can be defined by
 \[V_{\b\omega_1}=\bigoplus_{I\subset[m]}V_{\b\omega_1,I}.\]
 With this assumption on $V_{\b\omega_1}$, it is easily verified that for any $\beta\in V_{\b\omega_1}$, $\beta\alpha\neq 0$ (note that $\beta_j\b \omega_i\in \w H^*(K_{I_j\cup w_i})$).
Hence, if we can find an element $\lambda\in \mathrm{ann}(\b w_1)$,
such that for any element $\beta\in V_{\b\omega_1}\oplus\kk\cdot\lambda$, $\beta\alpha\neq 0$, then dim$_\kk(V_{\alpha})>\mathrm{dim}_\kk(V_{\b\omega_1})$, and so
\[\mathrm{dim}_\kk(\mathrm{ann}(\b \omega_1))>\mathrm{dim}_\kk(\mathrm{ann}(\alpha)).\]

Now let us find such an element. Suppose $\omega_1=\{v_3,v_4\}$ and $\omega_2=\{v_1,v_2\}$. Since $\omega_1\neq\omega_2$, we may assume $v_3\neq v_1,v_2$. Hence by Lemma \ref{lem:1}, there is an subset $I\subset[m]\setminus\{v_3\}$ such that $\omega_2\subset I$, $K_I$ is a triangulation of $S^1$ and $\w H^0(K_{J})\neq0$, where $J=\{v_3\}\cup I\setminus\omega_2$. If $v_4\in I\setminus\omega_2$, take $\lambda$ to be a generator of $\w H^0(K_{I\setminus\omega_2})=\kk$. A straightforward calculation shows that $\lambda\b\omega_2\neq0$ and $\lambda\in \mathrm{ann}(\b \omega_1)$ (since $\omega_1\cap (I\setminus\omega_2)\neq\varnothing$).
Since for any $\beta_1\in V_{\b\omega_1}$, there exists a subset $I_0\subset[m]$ containing $\omega_1$, such that $p_{I_0}(\beta_1\alpha)\neq0$, but $\lambda\alpha\in\w H^1(K_I)$ ($\omega_1\not\subset I$). Hence for any $\beta\in V_{\b\omega_1}\oplus\kk\cdot\lambda$, $\beta\alpha\neq 0$.

If $v_4\not\in I\setminus\omega_2$, let $i:K_{I\setminus\omega_1}\to K_{J}$ be the simplicial inclusion, and let $\lambda$ be a generator of $\w H^0(K_J)$ such that $i^*(\lambda)\neq0$.
It is easily verified that $\lambda\b\omega_2\neq0$ and $\lambda\in \mathrm{ann}(\b\omega_1)$. Hence
$\lambda\alpha\neq0\in\w H^1(K_{I\cup\{v_3\}})$. On the other hand, note that $V_{\b\omega_1}\subset\bigoplus_{\omega_1\cap U=\varnothing}\w H^*(K_U)$, so $V_{\b\omega_1}\cdot\b\omega_2\in\bigoplus_{v_3\not\in U}\w H^*(K_U)$, and so $p_{I\cup\{v_3\}}(V_{\b\omega_1}\cdot\b \omega_2)=0$. If $v_4\not\in\omega_2$, then clearly $p_{I\cup\{v_3\}}(V_{\b\omega_1}\cdot\omega_1)=0$; If $v_4\in \omega_2$, then the full subcomplex restricted on
$\{v_3\}\cup I\setminus\omega_1=I\setminus\{v_4\}$ is contractible, so $p_{I\setminus\{v_4\}}(V_{\b\omega_1})=0$ which implies that $p_{I\cup\{v_3\}}(V_{\b\omega_1}\cdot\omega_1)=0$. Hence on either case $p_{I\cup\{v_3\}}(V_{\b\omega_1}\cdot\alpha)=0$. It follows that  $\beta\alpha\neq 0$ for any $\beta\in V_{\b\omega_1}\oplus\kk\cdot\lambda$.
\end{proof}

\begin{cor}\label{cor:1}
If $K$ and $K'$ are both flag $2$-spheres without $4$-belt, and there is a graded isomorphism
\[\phi:H^*(\ZZ_K;\kk)\cong H^*(\ZZ_{K'};\kk),\]
Then for any missing face $\omega\in MF(K)$, $\phi(\b\omega)=\b\omega'$ (up to a multiplication) for some $\omega'\in MF(K')$.
Moreover for any $n$-belt $B_n$ of $K$ (denote by $\b B_n$ the generator of $\w H^1(B_n;\kk)=\kk$), $\phi(\b B_n)=\b B_n'$ (up to a multiplication) for some $n$-belt $B_n'\subset K'$.
\end{cor}
\begin{proof}
Since $\w\omega\in H^3(\ZZ_K;\kk)$ and $\phi$ is a graded isomorphism, then \[\phi(\b\omega)=\sum_{\omega'\in MF(K')}r_{\omega'}\cdot\b\omega',\ r_{\omega'}\in\kk.\]
The first assertion is equivalent to saying that there is exactly one $r_{\omega'}\neq0$. Suppose on the contrary, we may assume $\phi(\b\omega)=\b\omega'_1+\b\omega'_2$.
Thus one of $\w \omega'_i$, $i=1,2$, say $\w \omega'_1$ satisfies $\phi^{-1}(\b\omega'_1)=r\cdot\b\omega+r_0\cdot\b\omega_0$ with $r,r_0\neq0$, where $\omega_0\in MF(K)$.
According to Lemma \ref{lem:2},
\[\begin{split}
\mathrm{dim}_\kk(\mathrm{ann}_{R}(\b \omega))=\mathrm{dim}_\kk(\mathrm{ann}_{R'}(\b \omega'_1+\b\omega'_2))>\mathrm{dim}_\kk(\mathrm{ann}_{R'}(\b\omega'_1))\\
=\mathrm{dim}_\kk(\mathrm{ann}_{R}(r\cdot\b\omega+r_0\cdot\b\omega_0))>\mathrm{dim}_\kk(\mathrm{ann}_R(\b\omega)).
\end{split}\]
We get a contradiction.

For the second assertion, suppose $MF(B_n)=\{\omega_1,\dots,\omega_t\}$ (clearly $t=\binom{n}{2}-n$). Note that $\phi(\b B_n)=\sum_{i}r_i\cdot (\b B'_n)_i$, where $\{(B'_n)_i\}$ is the $n$-belt set of $K'$  (cf. the proof of \cite[Corollary 7.3]{FW}).
It is easy to see that $\b\omega_j$ ($1\leq j\leq t$) is a factor of $\b B_n$, and $\phi(\b\omega_j)=\w\omega'_j$  for some $\omega_j'\in MF(K')$. It follows that $\w\omega'_j$ is a factor of $(\b B'_n)_i$ whenever $r_i\neq0$. This implies that $\omega_j'\in MF(( B'_n)_i)$.
So we have \[MF((B_n')_{i})=\{\omega_1'\dots,\omega_t'\},\ \text{ whenever }\ r_i\neq0.\]
Since a simplicial complex is uniquely determined by its missing face set, then there is exactly one $r_i\neq0$, which is the second assertion.
\end{proof}
The next step is to distinguish a class $\b B_n\in H^{n+2}(\ZZ_K;\kk)$ for which $B_n$ is the link of a vertex of $K$ from another class
$\b{\mathfrak{B}}_n$ for which $\mathfrak{B}_n$ is an $n$-belt of $K$ but not the link of any vertex of $K$. That is the following:
\begin{lem}\label{lem:3}
If $K$ and $K'$ are both flag $2$-spheres without $4$-belt, and there is an isomorphism
\[\phi:H^*(\ZZ_K;\kk)\cong H^*(\ZZ_{K'};\kk),\]
then for any $n$-belt $B_n$ which is the link of a vertex of $K$,  $\phi(\b B_n)=\b B_n'$ (up to a multiplication) with $B_n'$ the link of a vertex of $K'$.
\end{lem}
\begin{proof}
As before, suppose $MF(B_n)=\{\omega_1,\dots,\omega_t\}$. Apparently, the $\kk$-subspace spanned by $\{\w\omega_1,\dots,\w\omega_t\}$ (denoted by $V_{B_n}$) is a maximal $3$-factor space of
$\w B_n$ in $H^*(\ZZ_K;\kk)$. Assume $B_n$ is the link of $v_{n+1}$ and the vertex sets of $B_n$ and $\mathrm{star}_K(v_{n+1})$ are resp. $[n]$ and $[n+1]$.
Since $K$ has no $4$-belt, then for any $v\in[m]\setminus[n+1]$, $MF(\mathrm{link}_K(v))\cap MF(B_n)=\varnothing$. It turns out that $K_{[n]\cup\{v\}}$ has the form $B_n\bigcup_{\sigma}\Delta^i$ ($i=0,1,2$), where $\{v\}\in\Delta^i$ and $\sigma\in B_n$. Thus $\w H^1(K_{[n]\cup\{v\}})\cong\kk$. Let $\xi_v\in H^{n+3}(\ZZ_K;\kk)$ be the element corresponding to a generator of $\w H^1(K_{[n]\cup\{v\}})$.
\cite[Proposition 3.2]{FW} implies that $V_{B_n}$ is also a maximal $3$-factor space of $\xi_v$, and it is easy to see that there are exactly $m-n-1$ linear independent elements
$\xi_{v_{n+2}},\dots,\xi_{v_m}\in H^{n+3}(\ZZ_K;\kk)$ with this property.

Suppose $MF(B_n')=\{\omega_1',\dots,\omega_t'\}$, and the vertex set of $B_n'$ is still $[n]$. From Lemma \ref{lem:2}, we have that $\phi(\b\omega_i)=\b\omega_i'$ (up to a permutation). Suppose on the contrary that $B_n'$ is not the link of any vertex of $K'$.
If for any $v\in[m]\setminus[n]$, $MF(\mathrm{link}_{K'}(v))\cap MF(B_n')=\varnothing$, then there are exactly $m-n$ linear independent elements
$\xi'_{v_{n+1}},\dots,\xi'_{v_m}\in H^{n+3}(\ZZ_{K'};\kk)$ such that $V_{B_{n}'}=\phi(V_{B_n})$ is a maximal $3$-factor space of $\xi'_{v_i}$ for each $n+1\leq i\leq m$.
This is impossible since $\phi$ is an algebra isomorphism. Hence there is at least one $v\in[m]\setminus[n]$
such that $MF(\mathrm{link}_{K'}(v))\cap MF(B_n')\neq\varnothing$.
An easy observation shows that there always exists an $l$-belt ($l\leq n$) $B_l'$ (with vertex set $\mathcal {L}$) in $K'$ such that $v\in \mathcal {L}$ and $\mathcal {L}\setminus\{v\}\subset[n]$. (see Figure \ref{fig:2} for an example).
\begin{figure}[!ht]
\vspace{6pt}
\begin{overpic}[scale=0.35]{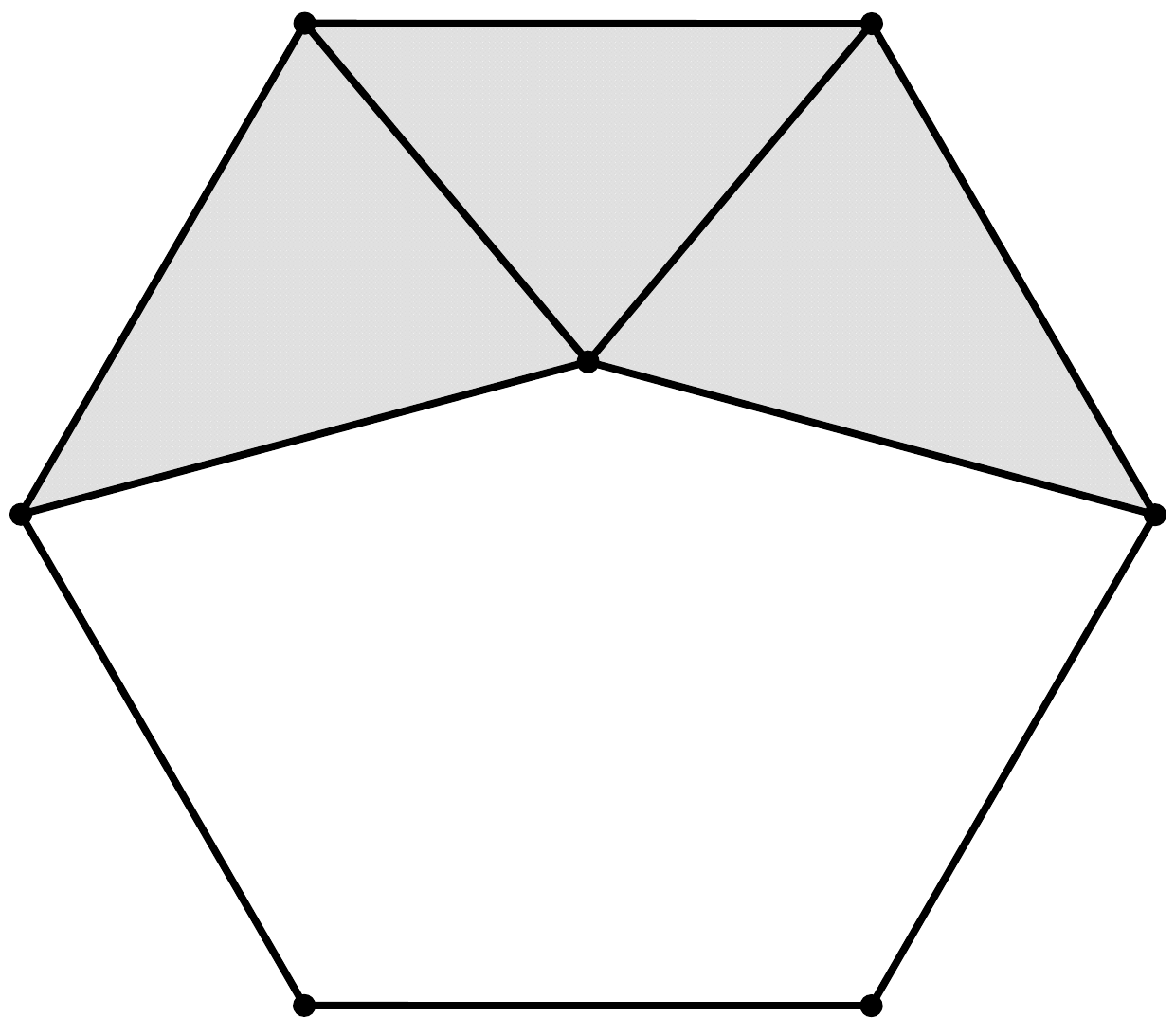}
\put(18,90){$v_5$}
\put(20,-5){$v_1$}
\put(48,50){$v$}
\put(74,-5){$v_2$}
\put(-8,42){$v_3$}
\put(100,42){$v_4$}
\put(74,89){$v_6$}
\end{overpic}
\vspace{6pt}
\caption{$n=6,\ l=5$ and $\mathcal {L}=\{v_1,v_2,v_3,v_4,v\}$}\label{fig:2}
\end{figure}
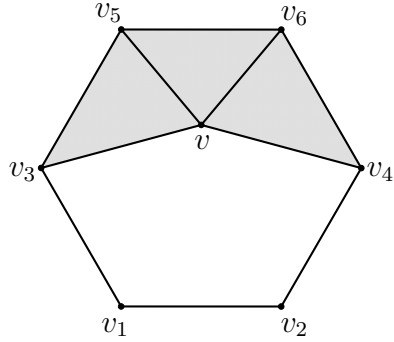
It is easy to see that $MF(B_l')\cap MF(B_n')$ contains $s=\binom{l}{2}-2l+3$ missing faces. We may assume $MF(B_l')\cap MF(B_n')=\{\omega'_1,\dots,\omega'_{s}\}$.
By Corollary \ref{cor:1}, $\phi^{-1}(\w B_l')=\w B_l$ for some $l$-belt $B_l$ in $K$, and $MF(B_l)\cap MF(B_n)=\{\omega_1,\dots,\omega_{s}\}$.
Together with the fact that $s=\binom{l}{2}-2l+3$, we can obtain that there is exactly one vertex $v$ of $B_l$, which is not in $[n+1]$.
Thus $MF(\mathrm{link}_K(v))\cap MF(B_n)\neq\varnothing$, and we get a contradiction. Therefore $B_n'$ has to be the link of a vertex of $K'$.
\end{proof}

Since a belt which is the link of a vertex can uniquely determine a vertex of $K$, thus combining all the results above, we can construct a $1$-to-$1$ correspondence between the vertex sets of $K$ and $K'$, $\psi:\Ss\to\Ss'$, which is induced by the isomorphism
\[\phi:H^*(\ZZ_K;\kk)\cong H^*(\ZZ_{K'};\kk).\]

Next we show that
\begin{lem}\label{lem:4}
The vertex sets correspondence $\psi$ above is actually a simplicial map, and so a simplicial isomorphism.
\end{lem}
\begin{proof}
We only need to verify that for each edge $e$ of $K$, $\psi(e)$ is an edge of $K'$ (since by flagness there is no $3$-belt in $K$ and $K'$).  An easy observation shows that if $(v_i,v_j)$ is an edge of $K$, then $B_{n_i}=\mathrm{link}_K(v_i)$ and $B_{n_j}=\mathrm{link}_K(v_j)$ have a common missing face $\omega$. Thus $\w\omega$ is a common factor of $\b B_{n_i}$ and $\b B_{n_j}$, and $\phi(\w\omega)$ is a common factor of
$\phi(\b B_{n_i})$ and $\phi(\b B_{n_j})$.
However if $\psi((v_i,v_j))=\{v_i',v_j'\}$ is a missing face of $K'$, then $MF(\mathrm{link}_{K'}(v_{i'}))\cap MF(\mathrm{link}_{K'}(v_{j'}))=\varnothing$  since
$K'$ has no $4$-belt (Otherwise if $\omega'$ is a common missing face then $K'_{\omega'\cup\{v_i',v_j'\}}$ is a $4$-belt). Set $B_{n_i}'=\mathrm{link}_{K'}(v_{i'})$ and $B_{n_j}'=\mathrm{link}_{K'}(v_{j'})$, then $\phi(\b B_{n_i})=\b B'_{n_i},\ \phi(\b B_{n_j})=\b B'_{n_j}$ by the definition of $\psi$. Note $\phi(\w\omega)=\w\omega'$ for some
$\omega'\in MF(K')$ by Corollary \ref{cor:1}, it follows from the fact $MF(B'_{n_i})\cap MF(B'_{n_j})=\varnothing$ that $\phi(\w\omega)$ is not a common factor of $\b B'_{n_i}$ and $\b B'_{n_j}$, a contradiction.
\end{proof}
\begin{proof}[Proof of Theorem \ref{thm:1}]
Since $K$ is a simplical $2$-sphere, $H^*(\ZZ_K)$ is torsion free, hence $H^*(\ZZ_K;\kk)\cong H^*(\ZZ_K)\otimes\kk$ for any field $\kk$.
It follows that $H^*(\ZZ_K;\kk)\cong H^*(\ZZ_{K'};\kk)$ for any field $\kk$. From \cite[Theorem 3.11 and Theorem 5.7]{FW}, a simplicial $2$-sphere $K$ is flag if and only if $\w {H}^*(\mathcal {Z}_{K})/([\mathcal {Z}_{K}])$ is a (nonzero) indecomposable ring. By \cite[Corollary 7.3]{FW}, for any simplicial complex $K$, there exists an element $\xi\in H^6(\ZZ_K;\kk)$ such that $\mathrm{ind}^3_R(\xi)=2$ if and only if there is a $4$-belt in $K$. Hence we finish the proof of Theorem \ref{thm:1} for the case $K'$ is a simplicial $2$-sphere.
\end{proof}
From the proof of this theorem we also obtain the following:
\begin{cor}
If $K$ is a flag $2$-sphere without $4$-belt, then any automorphism of $H^*(\ZZ_K)$ is induced by a combinatorial automorphism of $K$, i.e.,
\[Aut(H^*(K))\cong Aut(K).\]
\end{cor}
The general case that $K'$ is an arbitrary simplicial complex in Theorem \ref{thm:1} can be solved, once we prove the following:
\begin{thm}\label{thm:2}
Let $K$ be a simplical $2$-sphere, $K'$ an arbitrary simplicial complex. If there is a graded isomorphism:
\[H^*(\ZZ_K)\cong H^*(\ZZ_{K'}),\]
then $K'$ is also a simplicial $2$-sphere.
\end{thm}
To prove this, we need a generalization of the famous  \emph{Lower Bound Theorem} in combinatorial theory,  which was first proved by Barnette \cite{B73} for simplicial polytopes
\begin{LBT}[{\cite[Barnette]{B73}}]
Let $P$ be a simplicial $n$-polytope with $m$ vertices, and let $f_1$ be the number of edges of $P$. Then $f_1\geq mn-\binom{n+1}{2}$.
\end{LBT}
Actually by using some results about Gorenstein* complexes and following the line of Barnette's proof, we can generalize this result to (see this in a forthcoming paper on Generalized Lower Bound Theorem by F.~Fan)
\begin{thm}\label{thm:4}
Let $K$ be a $(n-1)$ dimensional Gorenstein* complexes with $m$ vertices, and let $f_1$ be the number of edges of $K$. Then $f_1\geq mn-\binom{n+1}{2}$.
\end{thm}
\begin{proof}[Proof of Theorem \ref{thm:2}]
According to Theorem \ref{thm:3}, $H^*(\ZZ_{K};\kk)\cong \mathrm{Tor}_{\kk[m]}(\kk(K),\kk)$ is a Poincar\'e algebra for any field $\kk$.
Since $H^*(\ZZ_K)\cong H^*(\ZZ_{K'})$, the argument in the proof of Theorem \ref{thm:1} shows that $H^*(\ZZ_K;\kk)\cong H^*(\ZZ_{K'};\kk)$ for any field $\kk$. Therefore $K'$ is also a Gorenstein* complex.
By using Theorem \ref{thm:4} and applying the same argument in the proof of \cite[Theorem 6.10]{FW}, we get that $K'$ is a simplical $2$-sphere.
\end{proof}
Now we give an application of Theorem \ref{thm:1} in combinatorial mathematics, which is related with chemistry and material science (cf. \cite{BE15} for more studies about this).
\begin{Def}
A (mathematical) fullerene is a simple 3-polytope with all facets pentagons and hexagons.
\end{Def}
\begin{thm}[{\cite[Buchstaber-Erokhovets]{BE15}}]
The dual simplicial $2$-sphere of any fullerene is flag and has no $4$-belt.
\end{thm}
From Theorem \ref{thm:1}, we immediately get the following
\begin{cor}
Every fullerene is $B$-rigid.
\end{cor}


\end{document}